\documentclass[a4paper,12pt,reqno]{amsart}

\usepackage{amsfonts,amssymb,amscd,amsmath,amsxtra,latexsym,amsbsy,enumerate,a4wide,verbatim}
\usepackage{mathrsfs}

\newtheorem{thm}{Theorem}[section]

\newtheorem{defn}[thm]{Definition}
\newtheorem{defprop}[thm]{Definition-Proposition}
\theoremstyle{definition}
\newtheorem{remark}[thm]{Remark}

\numberwithin{equation}{section}

\def\al{\alpha}

\def\ga{\gamma}
\def\de{\delta}

\def\si{\sigma}
\def\vp{\varphi}

\def\io{\mathrm{id}}

\def\De{\Delta}

\def\Up{\Upsilon}
\def\Z{\mathbb{Z}}

\def\C{\mathbb{C}}
\def\N{\mathbb{N}}

\def\cG{\mathcal G}

\def\scL{\mathscr L}

\newcommand{\rphis}[5]{\,_{#1}\vp_{#2} \left( \genfrac{.}{.}{0pt}{}{#3}{#4}
\ ;#5 \right)}

\title[$q$-Hankel transform]{A $q$-Hankel transform associated to the quantum linking groupoid for the
quantum $SU(2)$ and $E(2)$ groups}
\author{Kenny De Commer and Erik Koelink}
\date{\today}
\address{Department of Mathematics, University of Cergy-Pontoise,
UMR CNRS 8088, F-95000 Cergy-Pontoise, France}
\email{Kenny.De-Commer@u-cergy.fr}
\address{Radboud Universiteit Nijmegen, IMAPP, FNWI, Heyendaalseweg 135, 6525 AJ Nijmegen,
the Netherlands}
\email{e.koelink@math.ru.nl}

\begin{document}

\begin{abstract} A $q$-analogue of Erd\'elyi's formula for the Hankel transform of the product of
Laguerre polynomials is derived using the quantum linking groupoid between the quantum
$SU(2)$ and $E(2)$ groups. The kernel of the $q$-Hankel transform is given by the ${}_1\vp_1$-$q$-Bessel 
function, and then the transform of a product of two Wall polynomials times a $q$-exponential is calculated
as a product of two Wall polynomials times a $q$-exponential.
\end{abstract}

\maketitle


\section{Introduction}\label{sec:intro}

In 1938 Arthur Erd\'elyi \cite{Erde} proved the following formula for the Hankel transform of the product
of two Laguerre polynomials
\begin{equation}\label{eq:Erdelyi}
\begin{split}
\int_{0}^{\infty} x^{\nu} e^{-x^2} L_{m}^{(\nu - \sigma)}(x^2)L_{n}^{(\sigma)}(x^2) J_{\nu}(2xy) x\, d x
= \\
\frac{(-1)^{m+n}}{2} y^{\nu} e^{-y^2}L_m^{(\sigma -m +n)}(y^2) L_{n}^{(\nu -\sigma +m -n)}(y^2)
\end{split}
\end{equation} 
using integral representations for the Laguerre polynomials in case $n=m$, see \cite{KolbS} for a proof and historic overview. 
In \eqref{eq:Erdelyi} we require $n,m\in\N$, $\Re \nu>-1$, $y>0$, $\si\in\C$, and we use the standard notation for Laguerre polynomials and Bessel functions as
in e.g. \cite{Isma}. 
There are many relations between Laguerre polynomials and the Hankel transform, the most well-known being that the Laguerre polynomials 
arise as the eigenfunctions of the Hankel transform. 
For some of these identities there is a group theoretic interpretation, but as far as we know this is not the case for Erd\'elyi's identity \eqref{eq:Erdelyi}.

We present a $q$-analogue of Erd\'elyi's identity, which surprisingly \emph{does} have a quantum group theoretic derivation. 
The essential ingredient is the use of the linking quantum groupoid between the quantum group analogues of $SU(2)$ and the double cover of $E(2)$, 
the group of plane motions. This linking quantum groupoid is studied extensively in 
\cite{DeCo-CMP}, \cite{DeCo-AiM}, and we recall the necessary results in Section \ref{sec:qlinkinggroupoid}.

In some sense the result is reminiscent of the quantum group theoretic derivation of the addition formula for the little $q$-Legendre polynomials 
by Koornwinder \cite{Koor-SIAM1991} and Graf's addition formula for the ${}_1\vp_1$ $q$-Bessel functions \cite[\S 6]{Koel-DMJ}, which in turn is motivated by
\cite{Koor-SIAM1991}. The linking quantum groupoid allows us to connect these two developments, and we obtain the $q$-analogue 
of \eqref{eq:Erdelyi} from this connection.

The contents of the paper are as follows. In Sections \ref{sec:preliminaries-QSF} and \ref{sec:preliminaries-QGPS} we recall the necessary preliminaries on special functions and quantum groups, and we recall as well the definition of the standard Podle\'s sphere. In Section \ref{sec:qlinkinggroupoid} we present the above mentioned quantum linking groupoid and its comultiplication, which is employed in Section \ref{sec:qErdelyiformula} to derive
a $q$-analogue of \eqref{eq:Erdelyi}. Having the result at hand we can then extend various of the parameters to a more general domain.
The $q$-Hankel transform is as in \cite{KoorS}, and taking the inverse gives the same formula.
In Remark \ref{rmk:thmqErdelyi}(ii) we sketch an analytic, though verificational, proof.

Koornwinder's addition formula \cite{Koor-SIAM1991} has been generalized to Askey-Wilson polynomials in \cite{Koel-T} using Koornwinder's \cite{Koor-SIAM1993} twisted 
 primitive elements in quantum groups, and we expect that Theorem \ref{thm:qErdelyi} can be extended in a similar way to the level of the Askey-Wilson Hankel transform, see  \cite[Fig.~2.1]{KoelS}.
From \cite{DeCo-CMP} we know that there is also a quantum linking groupoid between the quantum group analogues of $SU(2)$ and 
the normalizer of $SU(1,1)$ in $SL(2,\C)$, the construction of which involves non-standard Podle\'{s} spheres. It might be possible to obtain a similar result in this case even though the 
unitary operator implementing the comultiplication becomes more involved. 


\section{Preliminaries: $q$-special functions}\label{sec:preliminaries-QSF}

In this section we recall the $q$-special functions needed, and we fix the notation.

\subsection{Basic hypergeometric series}\label{ssec:BHS}
We follow the notation for basic hypergeometric series as in Gasper and Rahman \cite{GaspR},
see also e.g. \cite{Isma}, \cite{KoekS}. Explicitly, we assume $0<q<1$ and define
the $q$-shifted factorials for $a\in\C$, $k\in\N$ by
\begin{equation*}
(a;q)_k \, = \, \prod_{i=0}^{k-1}(1-aq^i), \quad (a;q)_\infty \, = \, \lim_{k\to\infty} (a;q)_k,
\quad (a_1,\cdots,a_r;q)_k = \prod_{i=1}^r (a_i;q)_k
\end{equation*}
also allowing $k=\infty$ in the last definition.
Then the basic (or $q$-)hypergeometric series is
\begin{equation}\label{def:rphisbhs}
\rphis{r}{s}{a_1,\cdots,a_r}{b_1,\cdots b_s}{q, z} \, = \,
\sum_{k=0}^\infty \frac{(a_1,\cdots,a_r;q)_k}{(b_1,\cdots,b_s;q)_k} \frac{z^k}{(q;q)_k} \left( (-1)^k q^{\frac12 k(k-1)}\right)^{1+s-r}.
\end{equation}
We assume that $b_i\not\in q^{-\N}$ for all $i$.
The series terminates if $a_i\in q^{-\N}$ for some $i$. In general, the radius of convergence is
\begin{itemize}\item[$\bullet$] $1$ in
case $r=s+1$,
\item[$\bullet$] $\infty$ if $1+s>r$, and
\item[$\bullet$] $0$ if $1+s<r$.
\end{itemize}
 Note however that
\[
(b_1;q)_\infty \rphis{r}{s}{a_1,\cdots,a_r}{b_1,\cdots b_s}{q, z}
\]
is analytic in $b_1$, and there is no need to exclude $q^{-\N}$ for $b_1$. In case $b_1=q^{1-n}$, $n\in\N$,
the summation starts at $k=n$.

\subsection{Wall polynomials}
The Wall polynomials are defined as a subclass of the little $q$-Jacobi polynomials,
see \cite{GaspR},\cite{KoekS}. Explicitly, we define the Wall polynomials for $0<a<q^{-1}$ as
\begin{equation}\label{eq:defWallpols}
p_n(x;a;q) \, = \, \rphis{2}{1}{q^{-n}, 0}{aq}{q, qx} \, = \,
\frac{(-1)^n q^{\frac12n(n+1)}(ax)^n}{(aq;q)_n}\rphis{3}{2}{q^{-n},q^{-n}/a,1/x}{0,0}{q,q}
\end{equation}
see \cite[(2.8)]{Koor-SIAM1991}.
The Wall polynomials are $q$-analogues of the Laguerre polynomials in view
of
\[
\lim_{q\uparrow 1} p_n((1-q)x;q^\al;q) = \frac{L_n^{(\al)}(x)}{L_n^{(\al)}(0)},
\]
where the notation for Laguerre polynomials $L_n^{(\al)}(x)$ is the standard notation as in e.g. 
\cite[\S 4.6]{Isma}, \cite[\S 1.11]{KoekS}. 
The orthogonality relations and the dual orthogonality relations for the Wall polynomials are
\begin{equation}\label{eq:orthoanddualorthoWallpols}
\begin{split}
\sum_{k=0}^\infty \frac{(aq)^k}{(q;q)_k} p_n(q^k;a;q) p_m(q^k;a;q) \, = \, \de_{nm} \frac{(aq)^n (q;q)_n}{(aq;q)_n(aq;q)_\infty}, \\
\sum_{n=0}^\infty \frac{(aq;q)_n}{(aq)^n(q;q)_n} p_n(q^k;a;q) p_n(q^l;a;q) \, = \, \de_{kl}
\frac{(aq)^{-k} (q;q)_k}{(aq;q)_\infty},
\end{split}
\end{equation}
where the dual orthogonality relations correspond to the orthogonality relations for the
Al-Salam--Carlitz II polynomials,
see \cite{KoekS}, \cite[\S 2]{Koor-SIAM1991}, \cite[Prop.~3.3]{Lanc}.

\subsection{${}_1\vp_1$ $q$-Bessel functions}

For the ${}_1\vp_1$ $q$-Bessel functions (also known as Jackson's third $q$-Bessel function or as the Hahn-Exton $q$-Bessel function) we follow Koornwinder and Swarttouw \cite{KoorS}. Define the ${}_1\vp_1$ $q$-Bessel function of order $\nu$
\begin{equation}\label{eq:def1phi1qBesselfunction}
J_\nu (x;q) = x^\nu \frac{(q^{\nu+1};q)_\infty}{(q;q)_\infty} \rphis{1}{1}{0}{q^{\nu+1}}{q,qx^2}.
\end{equation}
Note that this is well-defined for $\nu=n\in\Z$, see \S \ref{ssec:BHS}, and in this case we have
$J_{-n}(w;q^2)= (-q)^n J_n(wq^n;q^2)$. We also have the symmetry
$J_\nu(q^\al;q^2) = J_\al(q^\nu;q^2)$, see \cite[Prop.~2.1]{KoorS}.
Moreover,
we have $\lim_{q\uparrow 1} J_\nu((1-q)x;q^2) = J_\nu(x)$, where $J_\nu(\cdot)$ is the Bessel function,
see \cite[\S 3]{KoorS}.
The ${}_1\vp_1$ $q$-Bessel functions can be obtained from the little $q$-Jacobi polynomials, see \cite[\S 3]{KoorS}, and
the following orthogonality relations hold:
\begin{equation}\label{eq:orthrellittleqBessel}
\sum_{k\in \Z}\, q^{k+n} J_{k+n}(q^l;q^2)\, q^{k+m} J_{k+m}(q^l;q^2)\, = \, \de_{nm}, \qquad n,m,l\in\Z.
\end{equation}
The orthogonality relations can also be rewritten as a $q$-Hankel transform pair:
\begin{equation}\label{eq:qHankeltransformpair}
g(q^n)\, = \, \sum_{k=-\infty}^\infty q^{2k}\, J_\nu(q^{k+n};q^2) \, f(q^k) \quad \Leftrightarrow
f(q^k)\, = \, \sum_{n=-\infty}^\infty q^{2n}\, J_\nu(q^{k+n};q^2) \, g(q^n)
\end{equation}
for $f$ and $g$ square integrable on $q^\Z$ with respect to the counting measure, $n,k\in\Z$, see \cite[\S 3]{KoorS}.
Note that \eqref{eq:qHankeltransformpair} can also be written using Jackson's $q$-integral,
\begin{equation}\label{eq:defJacksonqintegral}
\int_0^\infty f(x)\, d_qx \, = \, (1-q) \sum_{k=-\infty}^\infty f(q^k) q^k,
\end{equation}
for any function so that the sum converges.


\section{Preliminaries: quantum groups and Podle\'s sphere}\label{sec:preliminaries-QGPS}

In this section we introduce the quantum group analogues of $SU(2)$ and the double cover $\widetilde{E}(2)$ of the group of plane motions $E(2)$, as well as the standard Podle\'s sphere, on the level of von Neumann algebras. The main tools are the comultiplication and the corresponding action on the Podle\'s sphere. See \cite{DeCo-CMP}, \cite{DeCo-AiM} for precise information and references.

\subsection{The quantum $SU(2)$ group}\label{ssec:prelquantumSU2}

The quantum $SU(2)$ group is one of the earliest and most basic examples of a quantum group. It has
been introduced by Woronowicz \cite{Woro-RIMS1987}, see also \cite[\S 6.2]{Timm} and references.
The harmonic analysis and the relation with $q$-special functions on the quantum $SU(2)$ group is
of great interest.

\begin{defn}\label{def:quantumSU2}
Let $0<q<1$. The Hopf $\ast$-algebra $(\mathrm{Pol}(SU_q(2)),\De_+)$ is the universal unital $\ast$-algebra
generated by elements $\al$, $\ga$ which make the matrix $u = \begin{pmatrix} \al & -q\ga^\ast\\ \ga & \al^\ast\end{pmatrix}$
a unitary corepresentation.
\end{defn}

Definition \ref{def:quantumSU2} is a compact formulation due to Woronowicz's result \cite[Thm.~1.4]{Woro-RIMS1987}. In particular, 
the Hopf $\ast$-algebra $\mathrm{Pol}(SU_q(2))$ is generated by $\al$ and $\ga$ subject to the relations
\begin{equation}\label{eq:relationsinqSU2}
\al\ga=q\ga\al, \quad \al\ga^\ast=q\ga^\ast\al, \quad \ga^\ast\ga=\ga\ga^\ast, \quad 
\al^\ast\al+\ga^\ast\ga=1=\al\al^\ast+q^2 \ga^\ast\ga
\end{equation}
and with comultiplication $\De_+\colon \mathrm{Pol}(SU_q(2)) \to \mathrm{Pol}(SU_q(2))\otimes \mathrm{Pol}(SU_q(2))$ given by
\begin{equation}\label{eq:comultinqSU2}
\De_+(\al) = \al\otimes\al-q \ga^\ast \otimes \ga, \quad 
\De_+(\ga) = \ga\otimes\al+\al^\ast \otimes \ga.
\end{equation}

Associated to the Hopf $\ast$-algebra $(\mathrm{Pol}(SU_q(2)),\De_+)$ there is a von Neumann bialgebra
$(\scL^{\infty}(SU_q(2)),\De_+)$, which we now describe. We use the notation
$\scL(\Z)$ to denote the group von Neumann algebra of $\Z$,
i.e. the von Neumann algebra generated by the bilateral unitary shift operator $S\colon e_n\mapsto e_{n+1}$ acting
on $\ell^2(\Z)$, where the $e_n$ denote the standard orthonormal basis. In general, we will denote $e_{m,n,\ldots}$ for $e_m\otimes e_n\otimes \ldots$ 
as an element in a tensor product of $\ell^2$-spaces. 

The explicit implementation of the comultiplication $\De_+$ in the following Definition-Proposition \ref{defprop:SUq2asvNbialg} is  due
to Koornwinder \cite{Koor-SIAM1991}, Lance \cite{Lanc}. The formulation is taken from
\cite[App.~A]{DeCo-CMP}.

\begin{defprop}\label{defprop:SUq2asvNbialg}
The von Neumann bialgebra $(\scL^{\infty}(SU_q(2)),\De_+)$ has as its underlying  von Neumann algebra
$\scL^{\infty}(SU_q(2)) = B(\ell^2(\N))\bar{\otimes} \scL(\Z)$, where $\bar{\otimes}$ denotes the spatial tensor product of von Neumann algebras.
The comultiplication is given by  $\De_+(x) = W_+^\ast(1\otimes x)W_+$, where
$W_+$ is a unitary map defined by 
\begin{gather*}
W_+\colon (\ell^2(\N)\otimes \ell^2(\Z))\otimes (\ell^2(\N)\otimes \ell^2(\Z)) \rightarrow
(\ell^2(\Z)\otimes \ell^2(\Z))\otimes (\ell^2(\N)\otimes \ell^2(\Z)) \\
\xi_{r,s,p,t}^+\mapsto e_{r,s,p,t}, \qquad p\in\N, r,s,t\in\Z,
\end{gather*}
where $\{ \xi_{r,s,p,t}^+\mid p\in\N, r,s,t\in\Z\}$ is an orthonormal basis of $(\ell^2(\N)\otimes \ell^2(\Z))\otimes (\ell^2(\N)\otimes \ell^2(\Z))$
defined by 
\begin{gather*}
\xi_{r,s,p,t}^+ = \underset{v-w=t}{\sum_{v,w \in \N}} P^{+}(p,v,w) e_{v,r+p-w,w,s-p+v} \\
 P^{+}(p,v,w) =  (-q)^{p-w} q^{(p-w)(v-w)}
 \sqrt{\frac{(q^{2p+2},q^{2w+2};q^2)_\infty}{(q^{2v+2};q^2)_\infty}}
 \frac{(q^{2v-2w+2};q^2)_\infty}{(q^{2};q^2)_\infty}
 p_w(q^{2p};q^{2v-2w};q^2)
\end{gather*}
for $p,v,w\in \N$.
\end{defprop}

\begin{remark}\label{rmk1:defSUq2asvNbialg}
(i)  The fact that $\{\xi_{r,s,p,t}^+\}_{p\in\N, r,s,t\in\Z}$ is an orthonormal basis for the space
$(\ell^2(\N)\otimes \ell^2(\Z))\otimes (\ell^2(\N)\otimes \ell^2(\Z))$, i.e.~ that $W_+$ is a unitary operator, is equivalent to 
the orthogonality and dual orthogonality relations \eqref{eq:orthoanddualorthoWallpols} for the Wall polynomials, which can be stated as
\begin{equation}\label{eq:orthrelP^+}
\begin{split}
&\sum_{p=0}^\infty P^+(p,v,w)\, P^+(p,v',w') = \de_{v,v'} \qquad \text{if}\ v-w=v'-w', \\
&\underset{v-w=t}{\sum_{v,w \in \N}} P^{+}(p,v,w)\, P^+(p',v,w)\, =\, \de_{p,p'}.
\end{split}
\end{equation}
\\
(ii) 
Using the explicit expression \eqref{eq:defWallpols}
for the Wall polynomial as a ${}_3\vp_2$-function we find that $P^+(p,v,w)(-q)^{-p}$ is symmetric in
$p$, $v$ and $w$. 
The expression for $P^{+}(p,v,w)$ coincides with the notation as in \cite[Def.~0.4, App.~A]{DeCo-CMP}.
The identification with \cite[(2.5)]{Koor-SIAM1991} is
$P^{+}(p,v,w)=(-1)^p P_v(q^{2p};q^{2(w-v)}\mid q^2)$. \\
(iii) The embedding of $\mathrm{Pol}(SU_q(2))$ into $\scL^{\infty}(SU_q(2))$ is given by
\begin{equation}\label{eq:reprSUq(2)}
  \al e_{n,k} \,=\, \sqrt{1-q^{2n}}\, e_{n-1,k}, \qquad
\ga e_{n,k} \,=\, q^n\, e_{n,k+1},
 \end{equation}
with the convention $e_{-1}=0$ in $\ell^2(\N)$.
The fact that the comultiplication \eqref{eq:comultinqSU2}
agrees with Definition-Proposition \ref{defprop:SUq2asvNbialg} follows from suitable contiguous relations for the Wall polynomials, hence
the $P^+$-functions, see \cite[App.~A]{DeCo-CMP}.
(In 2004 Groenevelt (unpublished notes) obtained the same result.) 
For this the sign-difference with the notation of Koornwinder \cite{Koor-SIAM1991} is essential. \\
(iv) A straightforward calculation shows that with respect to the standard basis we have 
\[
W_+ \, e_{m,k,n,l} = \sum_{p\in \N} P^+(p,m,n)
e_{k+n-p,l-m+p,p,m-n}.
\]
\end{remark}


\subsection{The quantum $\widetilde{E}(2)$ group}\label{ssec:prelquantumE2}

\begin{defprop}\label{defprop:quantumE2}
The von Neumann bialgebra $(\scL^{\infty}(\widetilde{E}_q(2)),\De_0)$ has as its associated von Neumann algebra
$\scL^{\infty}(\widetilde{E}_q(2)) = B(\ell^2(\Z)) \bar{\otimes}\scL(\Z)$.
The comultiplication $\De_0(x) = W_0^*(1\otimes x)W_0$ for all $x\in \scL^{\infty}(\widetilde{E}_q(2))$ is defined by
the unitary map 
\begin{gather*}
W_0\colon (\ell^2(\Z)\otimes \ell^2(\Z))\otimes (\ell^2(\Z)\otimes \ell^2(\Z))\rightarrow (\ell^2(\Z)\otimes \ell^2(\Z))\otimes (\ell^2(\Z)\otimes \ell^2(\Z)) \\
\xi_{r,s,p,t}^0\mapsto e_{r,s,p,t}, 
\end{gather*}
where $\{\xi_{r,s,p,t}^0\mid r,s,p,t\in\Z\}$ is the orthonormal basis of  $(\ell^2(\Z)\otimes \ell^2(\Z))\otimes (\ell^2(\Z)\otimes \ell^2(\Z))$ 
defined by
\begin{equation*}
\xi_{r,s,p,t}^0 = \underset{v-w=t}{\sum_{v,w\in \Z}} P^{0}(p,v,w) e_{v,r+p-w,w,s-p+v}
\end{equation*}
for $p,r,s,t\in \Z$, where
$P^{0}(p,v,w)\, = \, (-q)^{p-w} \, J_{v-w}(q^{p-w};q^2)$ for $p,v,w\in\Z$.
\end{defprop}

\begin{remark} (i) 
Using the explicit expression \eqref{eq:def1phi1qBesselfunction} and the symmetries for the
${}_1\vp_1$ $q$-Bessel functions as in \cite[Prop.~2.1, (2.6)]{KoorS} we find that $P^0(p,v,w)(-q)^{-p}$ is symmetric in
$p$, $v$ and $w$.
The expression for $P^{0}(p,v,w)$ coincides with the notation as in \cite[Prop.~0.6, App.~A]{DeCo-CMP}.
From \eqref{eq:orthrellittleqBessel} and the symmetries for the ${}_1\vp_1$ $q$-Bessel functions we find
that \eqref{eq:orthrellittleqBessel} gives
\begin{equation}\label{eq:orthrelP^0}
\begin{split}
&\sum_{p=-\infty}^\infty P^0(p,v,w)\, P^0(p,v',w') = \de_{v,v'} \qquad \text{if}\ v-w=v'-w', \\
&\underset{v-w=t}{\sum_{v,w \in \Z}} P^0(p,v,w)\, P^0(p',v,w)\, =\, \de_{p,p'}.
\end{split}
\end{equation}
Note that \eqref{eq:orthrelP^0} are equivalent to $\{\xi_{r,s,p,t}^0\}_{r,p, s,t\in\Z}$ being an orthonormal basis, hence $W_0$ is a unitary operator.
See \cite[Prop.~4.2, App.~A]{DeCo-CMP} for the implementation of the comultiplication, and also
\cite{Koel-DMJ} where the implementation is implicit.

\noindent
(ii) Define
\begin{equation}\label{eq:defvinEq}
ve_{m,k} = e_{m-1,k}, \qquad v\in \scL^{\infty}(\widetilde{E}_q(2)),
\end{equation}
and let $n$ denote the unique unbounded normal operator affiliated with  $\scL^{\infty}(\widetilde{E}_q(2))$ satisfying 
\[n e_{m,k}= q^m e_{m,k+1}.\] 
Then $v$ and $n$ are the generators of the quantum group $\widetilde{E}_q(2)$ as introduced by Woronowicz \cite{Woro-LMP1991}. In particular,
\begin{equation}\label{eq:defDeltavinEq}
\De_0(v) = W_0^\ast (1\otimes v) W_0 = v\otimes v.
\end{equation}
as follows from $P^0(p,v,w)= P^0(p-1,v-1,w-1)$, see \cite[App. A]{DeCo-CMP} for more information.
\\
\noindent
(iii) Note that $\lim_{N\to \infty} P^+(p+N,v+N,w+N)= P^0(p,v,w)$, and then 
\eqref{eq:orthrelP^+} go over into \eqref{eq:orthrelP^0}, 
cf. \cite[\S3, App.~A]{KoorS}. One may consider this limit as a reflection of the contraction
procedure of the quantum $SU(2)$ group to the quantum $\tilde{E}(2)$ group, see 
\cite[\S 2]{Woro-CMP1992}. 
\end{remark}


\subsection{The standard Podle\'s sphere}\label{ssec:prelPodlessphere}

The Podle\'s spheres are a $1$-parameter family of quantum analogues of the $2$-sphere originally introduced in \cite{Podl}.
We need one particular case, see \cite{DeCo-CMP} for more details.

\begin{defn}\label{def:Podlessphere}
Let $0<q<1$. The $^*$-algebra $\mathrm{Pol}(S_{q}^2)$ is generated by elements $X$, $Y$ and $Z$
satisfying $X^\ast= Y$, $Z^\ast= Z$ and
\[
XZ = q^2ZX, \quad YZ = q^{-2}ZY, \quad XY = Z -q^2 Z^2,\quad YX = Z -q^{-2} Z^2.
\]
Then $\mathrm{Pol}(S_{q}^2)$ is the polynomial $\ast$-algebra associated to the
\emph{standard Podle\'{s} sphere} $S_q^2$. The algebra
$\mathrm{Pol}(S_{q}^2)$ carries a coaction $\Up\colon \mathrm{Pol}(S_{q}^2)\to
\mathrm{Pol}(SU_q(2))\otimes \mathrm{Pol}(S_{q}^2)$, determined by
\[
\Up \begin{pmatrix} X \\ 1-(1+q^2)Z \\ Y \end{pmatrix} =
\begin{pmatrix} \alpha^2 & -\gamma^*\alpha & -q(\gamma^*)^2 \\ (1+q^2)\alpha\gamma & 1-(1+q^2)\gamma^*\gamma & (1+q^2)\gamma^*\alpha^* \\ -q\gamma^2 & -\alpha^*\gamma & (\alpha^*)^2\end{pmatrix}
\otimes \begin{pmatrix} X \\ 1-(1+q^2)Z \\Y\end{pmatrix}.
\]
\end{defn}

\begin{remark}\label{rmk:Podlessphere}
(i) The map $\phi\colon\mathrm{Pol}(S_{q}^2)\to \mathrm{Pol}(SU_q(2))$ defined by
\[
X \mapsto -\ga^\ast\al \qquad Z \mapsto \ga^*\ga \qquad Y \mapsto  -\al^\ast\ga
\]
embeds $\mathrm{Pol}(S_{q}^2)$ as a left coideal in $\mathrm{Pol}(SU_q(2))$.
The embedding is equivariant: $(\io\otimes \phi)\Up = \De_+\phi$.\\
(ii) The operators on $\ell^2(\N)$ defined by
\begin{equation}\label{eq:reprPodles}
X\,e_n\, =\, -q^{n-1}\sqrt{1-q^{2n}}\, e_{n-1}, \quad
Z\, e_n\, =\, q^{2n}\, e_{n}, \quad
Y\, e_n\, =\, -q^{n}\sqrt{1-q^{2n+2}}\, e_{n+1}.
\end{equation}
represent $\mathrm{Pol}(S_{q}^2)$ faithfully. Set $U\colon e_{n,k} \mapsto e_{n,n+k}$,
so that $U\in B(\ell^2(\N)\otimes\ell^2(\Z))$ is unitary. Then \eqref{eq:reprPodles} and
\eqref{eq:reprSUq(2)} are related by $U$, so 
$U(x\otimes 1)U^\ast = \phi(x)$ for all $x\in \mathrm{Pol}(S_{q}^2)$.
\end{remark}


\section{Quantum linking groupoid}\label{sec:qlinkinggroupoid}

The quantum linking groupoid relates the quantum group analogues of $SU(2)$ and the double cover of  $E(2)$
as given in Definition-Propositions \ref{defprop:SUq2asvNbialg} and \ref{defprop:quantumE2}. In order to describe
the construction briefly we start with an explicit realisation of the Podle\'s sphere,
see \cite{DeCo-CMP}, \cite{DeCo-AiM} for more information.


\subsection{Implementation of the Podle\'s sphere and quantum linking groupoid}\label{ssec:implcorepPodlessphere}

We define
\[
N=B(\ell^2(\N),\ell^2(\Z))\bar\otimes \scL(\Z)\subset B(\ell^2(\N)\otimes\ell^2(\Z),\ell^2(\Z)\otimes\ell^2(\Z))
\]
and we equip it with the normal linear map
\begin{equation}\label{eq:defcomultN}
\De_{0+} \colon N\to N\bar\otimes N, \qquad
\De_{0+}(x) \, = \, W_0^\ast (1\otimes x) W_+ \quad x\in N,
\end{equation}
which is easily seen to be well-defined using the bicommutant theorem.

Then $\Delta_{0+}$ can be shown to be coassociative, and to turn $N$ into a `linking bimodule coalgebra' between  $\scL^{\infty}(\tilde{E}_q(2))$ and $\scL^{\infty}(SU_q(2))$. Together with the space of adjoint operators and the function algebras on the two quantum groups, it can be considered as a quantum groupoid, called `linking quantum groupoid', see \cite{DeCo-CMP}, \cite{DeCo-AiM}.

\begin{defprop}\label{prop:implementationUpbycG}\cite[\S 3]{DeCo-AiM} 
Define the unitary operator
\begin{gather*}
\cG \colon \ell^2(\N)\otimes \ell^2(\Z) \otimes \ell^2(\N) \to
\ell^2(\Z)\otimes \ell^2(\Z) \otimes \ell^2(\N) \\
\cG \colon \eta_{r,p,t} \mapsto e_{t-p,r,p}, \quad
\eta_{r,p,t} =  \underset{v-w=t}{\sum_{v,w\in \N}} P^+(p,v,w) e_{v,r+p-w,w},
\end{gather*} where $p\in \N$, $r,t\in \Z$.
Then $\cG\in N\bar{\otimes} B(\ell^2(\N))$ and $\cG$ implements
$\Up\colon \mathrm{Pol}(S_{q}^2)\to
\mathrm{Pol}(SU_q(2))\otimes \mathrm{Pol}(S_{q}^2)$, i.e.
\[
\Up(x) \, = \, \cG^\ast (1\otimes x)\cG, \qquad \forall\, x\in \mathrm{Pol}(S_{q}^2).
\]
Moreover, $(\De_{0+} \otimes \io)(\cG) = \cG_{13}\cG_{23}$.
\end{defprop}

The last formula of Proposition \ref{prop:implementationUpbycG} means 
that $\cG$ is to be seen as a unitary projective corepresentation of $\scL^{\infty}(SU_q(2))$. 
Together with the explicit implementation of the 
comultiplication $\De_{0+}$ it is the crucial formula for this paper. In Section \ref{sec:qErdelyiformula} this identity is
converted into a $q$-analogue of Erd\'elyi's formula. This formula follows from the general construction as described
in \cite[\S 2]{DeCo-AiM} introduced to study Morita equivalence of quantum groups. 

We can expand $\cG$ by
\begin{equation}\label{eq:expaGinGnm}
\cG = \sum_{n,m\in\N} \cG_{mn} \otimes e_{mn}, \qquad
\cG_{mn}\colon e_{a,b} \mapsto P^+(m,a,n) e_{a-n-m,b+n-m}
\end{equation}
where $e_{mn} \in B(\ell^2(\N))$ is the matrix-unit defined by $e_{mn}e_k = \de_{nk}e_m$.

Finally, we define
\begin{equation}\label{eq:defGlmn}
\cG^{(l)}_{mn} = v^l \cG_{mn} \in B(\ell^2(\N)\otimes \ell^2(\Z),\ell^2(\Z)\otimes \ell^2(\Z)),
\qquad \cG^{(l)} = \sum_{m,n\in\N} \cG_{mn}^{(l)}\otimes e_{mn}
\end{equation}
where $v$ is the shift operator \eqref{eq:defvinEq}. Then Proposition \ref{prop:implementationUpbycG} and
\eqref{eq:defDeltavinEq} imply
\begin{equation}\label{eq:DeoncGl}
 (\De_{0+} \otimes \io)(\cG^{(l)}) = \cG_{13}^{(l)}\cG_{23}^{(l)}
\end{equation}
Moreover, from \eqref{eq:expaGinGnm} and \eqref{eq:defvinEq} we find
\begin{equation}\label{eq:actionGlmn}
\cG^{(l)}_{mn}\colon \ell^2(\N) \otimes \ell^2(\Z) \to \ell^2(\Z) \otimes \ell^2(\Z), \quad
\cG^{(l)}_{mn}\colon e_{a,b} \mapsto P^+(m,a,n) e_{a-l-n-m,b+n-m}.
\end{equation}


\section{A $q$-analogue of Erd\'{e}lyi's formula}\label{sec:qErdelyiformula}

The goal of this section is to prove a first $q$-analogue of Erd\'elyi's formula
\eqref{eq:Erdelyi} on the lowest level of the $q$-scheme as discussed in
\cite[Fig.~1.2, \S 4]{KoelS}.

\begin{thm}\label{thm:qErdelyi}
Recall the notation of \eqref{eq:defWallpols}, \eqref{eq:def1phi1qBesselfunction} for Wall polynomials and little $q$-Bessel functions. Write $\tilde{p}_n(x;a;q) = (qa;q)_n \;p_n(x;a;q)$. 
Then we have
\begin{equation*}
\begin{split}
&\sum_{p=0}^\infty q^{2p} q^{p\nu} (q^{2+2p};q^2)_\infty \; \tilde{p}_n(q^{2p};q^{2\si};q^2)\; \tilde{p}_m(q^{2p};q^{2(\nu-\si)};q^2)  \; J_{\nu}(zq^{p};q^2) \\
& = (-q)^{n+m} z^{\nu} q^{(m-n)^2} q^{2n\si+2m(\nu-\si)}
(z^2q^{2(1+n+m)};q^2)_\infty \\ & \qquad \times \tilde{p}_n(z^2q^{2(n+m)};q^{2(\nu-\si+m-n)};q^2)\;\tilde{p}_m(z^2q^{2(n+m)};q^{2(\si+n-m)};q^2)
\end{split}
\end{equation*}
where $\Re \nu>-1$, $z\in \C$ with $|\arg z|<\pi$, $n,m\in \N$, $\si\in\C$.
\end{thm}

\begin{remark}\label{rmk:thmqErdelyi}
(i) Write $E_{q^2}(z) = (-z;q^2)_\infty$ for the big $q$-exponential function \cite[\S 1.3]{GaspR}, and write $\check{p}_n(x;a;q) = (qa;q)_{\infty}p_n(x;a;q)$.
 Then using the notation \eqref{eq:defJacksonqintegral}, we can rewrite the result of Theorem \ref{thm:qErdelyi} as
\begin{equation}\label{eq:qErdelyiqintversion}
\begin{split} 
&\frac{1}{1-q}\int_{0}^\infty x^{\nu} E_{q^2}(-q^2x^2)\; \check{p}_n(x^2;q^{2\si};q^2)\check{p}_m(x^2;q^{2(\nu-\si)};q^2) J_{\nu}(zx;q^2)\,xd_qx  \\
& =  (-q)^{n+m} z^{\nu} q^{(m-n)^2} q^{2n\si+2m(\nu-\si)}
E_{q^2}(-z^2q^{2(1+n+m)}) \\ & \qquad \times \check{p}_n(z^2q^{2(n+m)};q^{2(\nu-\si+m-n)};q^2)\check{p}_m(z^2q^{2(n+m)};q^{2(\si+n-m)};q^2)
\end{split}
\end{equation}
Replacing $x$, $z$ by $x\sqrt{1-q}$, $z\sqrt{1-q}$ and using the $q$-Gamma function, we see that we can take the formal limit $q\uparrow 1$ of \eqref{eq:qErdelyiqintversion} to \eqref{eq:Erdelyi}.
\par\noindent
(ii) The special case $n=m=0$ is
\begin{equation}
\sum_{p=0}^\infty q^{2p} q^{p\nu} (q^{2+2p};q^2)_\infty  J_{\nu}(zq^{p};q^2)
 =  z^{\nu}
(z^2q^{2};q^2)_\infty
\end{equation}
which is directly proved using the $q$-binomial formula and the big $q$-exponential function,
see \cite[\S 1.3]{GaspR}, after using \eqref{eq:def1phi1qBesselfunction} and interchanging summations.
In the same spirit we can evaluate the left hand side of Theorem \ref{thm:qErdelyi} in the case $m=0$, which gives a series expansion in $z$ (up to $z^\nu$). Expanding the right hand side of Theorem
\ref{thm:qErdelyi} in the same way, we see that we need the $q$-Chu-Vandermonde summation
\cite[(1.5.3)]{GaspR} to establish equality.
The case for arbitrary degrees $m$ can be obtained by induction with respect to $m$, and using the three-term recurrence relation for the Wall polynomials both in the degree $n$ as in $m$.
Note that this is a verificational proof, and completely different in spirit from e.g. Rahman's \cite{Rahm} analytic
proof of Koornwinder's \cite{Koor-SIAM1991} addition formula for the little $q$-Legendre polynomials.
It would be desirable to have an analytic proof of Theorem \ref{thm:qErdelyi}
in the style of Rahman \cite{Rahm}. 
\\
\noindent
(iii) We need the additional freedom of $l$ introduced in \eqref{eq:defGlmn}
in order to get the result of Theorem \ref{thm:qErdelyi} in full generality. \\
\noindent
(iv) Theorem \ref{thm:qErdelyi} can be viewed as $q$-Hankel transform. The inverse, see \cite[(3.4)]{KoorS}, can be written down, and the resulting inverse identity is equivalent to Theorem \ref{thm:qErdelyi} for $z=q^r$, $r\in\Z$.
\end{remark}

\begin{proof}[Proof of Theorem \ref{thm:qErdelyi}.]
We start with the corepresentation property of $\cG^{(l)}$, see \eqref{eq:DeoncGl}, in combination with the implementation of the comultiplication as in Proposition \ref{prop:implementationUpbycG} to find the operator identity
\begin{equation}\label{eq:qErdelyioperator}
W_{0,12}^\ast \cG^{(l)}_{23} W_{+,12} = \cG^{(l)}_{13} \cG^{(l)}_{23}
\colon \bigl( \ell^2(\N)\otimes\ell^2(\Z)\bigr)^{\otimes 2} \otimes \ell^2(\N)
\to \bigl( \ell^2(\Z)\bigr)^{\otimes 4}\otimes \ell^2(\N)
\end{equation}
Since \eqref{eq:qErdelyioperator} is an identity for operators, we let \eqref{eq:qErdelyioperator}
act on the basis vector
$e_{a,b,c,d,e}$, $a,c,e\in \N$, $b,d\in \Z$.
The action of the left hand side of \eqref{eq:qErdelyioperator} on this basis vector can be calculated 
using Remark \ref{rmk1:defSUq2asvNbialg}(iv), \eqref{eq:actionGlmn} and Definition-Proposition \ref{defprop:quantumE2}, and we find
\begin{equation}\label{eq:qErdelyioperatorLHSonbasis}
\sum_{p,m\in\N} P^+(p,a,c) P^+(m,p,e) \xi^0_{b+c-p,d-a+p,p-l-m-e,a-c+e-m} \otimes e_m,
\end{equation}
with $\xi^0_{rspt}$ as in Definition-Propostion \ref{defprop:quantumE2}.
The action of the right hand side of \eqref{eq:qErdelyioperator} on this basis vector can be calculated using \eqref{eq:actionGlmn} and we find
\begin{equation}\label{eq:qErdelyioperatorRHSonbasis}
\sum_{m,r\in\N} P^+(m,c,e) P^+(r,a,m)
e_{a-l-r-m,b-r+m,c-l-e-m,d+e-m,r}.
\end{equation}

Since \eqref{eq:qErdelyioperator} leads to the equality of \eqref{eq:qErdelyioperatorLHSonbasis} and \eqref{eq:qErdelyioperatorRHSonbasis} as identity in $\ell^2(\Z)^{\otimes 4}\otimes \ell^2(\N)$, we can take
inner products with an arbitrary basisvector $e_{u,v,w,x,y} \in
\ell^2(\Z)^{\otimes 4}\otimes \ell^2(\N)$ to obtain a scalar identity.
A calculation shows that the inner product with \eqref{eq:qErdelyioperatorRHSonbasis} gives
\begin{equation*}
P^+(d+e-x,c,e) P^+(y,a,d+e-x)
\end{equation*}
in case
\begin{equation}\label{eq:qErdelyiconditionsRHS}
d+e-x=c-l-e-w=-b+y+v=a-l-y-u\in \N
\end{equation}
and $0$ otherwise. A calculation shows that the inner product with \eqref{eq:qErdelyioperatorLHSonbasis} gives
\begin{gather*}
\sum_{p=0}^\infty P^+(p,a,c) P^+(y,p,e) P^0(p-l-y-e, a-c+e-y+w,w).
\end{gather*}
in case $u-w=a-c+e-y$, $v+w=b+c-l-y-e$, and $x-u=d+y+e-a+l$ and $0$ otherwise. Assuming these conditions, which match the equalities in \eqref{eq:qErdelyiconditionsRHS}, and eliminating superfluous parameters we obtain
\begin{equation}\label{eq:qErdroughform1}
\begin{split}
&\sum_{p=0}^\infty P^+(p,a,c) P^+(y,p,e) P^0(p-l-y-e, a-c+e-y+w,w)  \\
& = \begin{cases} P^+(c-l-e-w,c,e) P^+(y,a,c-l-e-w), & c-l-e-w \in \N \\
0, & c-l-e-w \in \Z_{<0}
\end{cases}
\end{split}
\end{equation}
Using the symmetry in Remark \ref{rmk1:defSUq2asvNbialg}(ii), we can write
$P^+(y,p,e)=(-q)^{y-p} P^+(p,e,y)$,
$P^+(y,a,c-l-e-w)=(-q)^{l+e+w+y-c}P^+(c-l-e-w,a,y)$ and $P^+(c-l-e-w,c,e) = P^+(c-l-e-w,e,c)$. 
Next use the explicit expressions of $P^+$ and $P^0$ in terms of Wall polynomials and $q$-Bessel
functions to find 
\begin{equation*}
\begin{split}
&\sum_{p=0}^\infty q^{2p} q^{p(a-c+e-y)} (q^{2+2p};q^2)_\infty p_c(q^{2p};q^{2a-2c};q^2) p_y(q^{2p};q^{2e-2y};q^2) J_{a-c+e-y}(q^{p-l-y-e-w};q^2) \\
& = C (-q)^{y+c} (q^{2+2(c-l-e-w)};q^2)_\infty p_c(q^{2(c-l-e-w)};q^{2e-2c};q^2) p_y(q^{2(c-l-e-w)};q^{2a-2y};q^2), \end{split}
\end{equation*}
where \[
C = q^{c(a-c)} q^{y(e-y)} q^{(-l-e-w)(e-c)} q^{(c-l-e-w-y)(a-y)}
\frac{(q^{2+2e-2c}, q^{2+2a-2y};q^2)_\infty}{(q^{2+2a-2c},q^{2+2e-2y};q^2)_\infty}.
\]
Replacing $(c,a-c,y,e-y,-l-y-e-w)$ by $(n,\si,m,\nu-\si,z')$ we obtain the result for
$z=q^{z'}$, $z'\in\Z$, $\nu,\si\in\Z$ after a straightforward calculation.

Replacing $q^{z'}$ by $z$, and removing $z^\nu$ from both sides, we see that we can extend by analytic continuation in $z$. Similarly, we can extend analytically in $\nu$ to $\Re \nu>-1$.
Since the resulting identity is a Laurent polynomial in $q^{2\si}$ and the convergence is uniform in $q^{2\si}$ on bounded sets we find that the resulting identity is valid for arbitrary $\si\in\C$.
\end{proof}

\medskip
\textbf{Acknowledgement.} We thank the referees for useful comments, which have improved the readability of the paper. 


\end{document}